\documentclass[1p]{elsarticle}

\usepackage{lineno,hyperref}
\usepackage{amsmath,amsthm,amsfonts, amssymb}
\usepackage{mathtools}
\usepackage{algorithm}
\usepackage[noend]{algpseudocode}
\usepackage{subfigure}
\usepackage{color}
\usepackage{graphicx}
\usepackage{caption}
\usepackage{float}

\usepackage[
    backgroundcolor=orange
]{todonotes}

\modulolinenumbers[5]

\newcommand{\field}[1]{\mathbb{#1}}
\newcommand{\R}{\field{R}}

\newtheorem{theorem}{Theorem}
\newtheorem{remark}{Remark}[section]

\makeatletter
\newcommand*{\rom}[1]{\expandafter\@slowromancap\romannumeral #1@}
\makeatother

\begin{document}

\graphicspath{{figures/}}

\begin{frontmatter}

\title{Stochastic Deep-Ritz for Parametric Uncertainty Quantification}


%
%
%
%

\author[1,2]{Ting Wang}
\author[2]{Jaroslaw Knap}

\address[1]{
  Booz Allen Hamilton Inc., McLean, VA 22102, USA}

\address[2]{
  Physical Modeling and Simulation Branch, DEVCOM Army Research Laboratory,
  Aberdeen Proving Ground, MD, 21005-5066, USA}

\begin{abstract}
  Scientific machine learning has become an increasingly important
  tool in materials science and engineering. It is particularly well
  suited to tackle material problems involving many variables or to
  allow rapid construction of surrogates of material models, to name
  just a few. Mathematically, many problems in materials science and
  engineering can be cast as variational problems. However, handling of
  uncertainty, ever present in materials, in the context of
  variational formulations remains challenging for scientific machine
  learning.  In this article, we propose a deep-learning-based
  numerical method for solving variational problems under
  uncertainty. Our approach seamlessly combines deep-learning
  approximation with Monte-Carlo sampling. The resulting numerical
  method is powerful yet remarkably simple. We assess its performance
  and accuracy on a number of variational problems.
\end{abstract}


\end{frontmatter}



\section{Introduction}
Randomness and uncertainty are ubiquitous in materials science and
engineering. In order to facilitate analysis, and ultimately design,
of materials, it is essential to accurately quantify the uncertainty
in material models. A good case in point is additive manufacturing of
materials. Additive manufacturing commonly involves numerous sources
of uncertainty such as laser power, material composition or particle
size.  The uncertainty ultimately influences the material
microstructure and hence, in turn, affects the mechanical performance
of manufactured parts.  Therefore, it is critical to take the
uncertainty into account as to clear the way for robust additive
manufacturing.  From the mathematical modeling perspective, the
uncertainty is frequently modeled as the random input data in the form
of a random field $\kappa$ to a (stochastic) variational problem
\begin{equation}
  \label{eqn:abstract-energy}
  \min_{u} \mathbb{E}\left[\int_D I(x,  u,  \nabla u; \kappa) \,  dx   \right]
\end{equation}
for a Lagrangian $I$ depending on the random field $\kappa$.

Traditionally, solutions of~\eqref{eqn:abstract-energy} are sought by
first computing the associated strong form and then solving the
resulting differential equation by numerical methods for stochastic
differential equations. A great number of computational methodologies
have been developed to solve stochastic differential equations.  Among
them, stochastic collocation (SC), stochastic Galerkin (SG), and Monte
Carlo (MC)/quasi Monte Carlo (QMC) are the three mainstream
approaches.  SC aims to first solve the deterministic counterpart of a
stochastic differential equation on a set of collocation points and
then interpolate over the entire image space of the random
element~\cite{babuvska2007stochastic,nobile2008sparse,
  nobile2008anisotropic}.  Hence, the method is non-intrusive, meaning
that it can take advantage of existing solvers developed for
deterministic problems.  Similarly, MC is non-intrusive since it
relies on taking sample average over a set of deterministic solutions
computed from a set of realizations of the random
field~\cite{babuska2004galerkin, matthies2005galerkin,
  kuo2016application}.  In contrast, SG is considered as intrusive
since it requires a construction of discretization of both the
stochastic space and physical space simultaneously and, as a result,
it commonly tends to produce large systems of algebraic equations
whose solutions are needed~\cite{babuska2004galerkin,
  ghanem2003stochastic, gunzburger2014stochastic, xiu2002wiener,
  xiu2003modeling}. However, these algebraic systems differ
appreciably from their deterministic counterparts and thus existing
deterministic solvers can rarely be utilized.

It is important to stress that all traditional numerical methods for
stochastic differential equations suffer from the curse of
dimensionality and hence tend to be limited to low dimensional
problems (e.g., $D \subset \mathbb{R}^2$ or $D \subset \mathbb{R}^3$).
Recently, deep neural networks (DNN) have garnered some acceptance in
computational science and engineering, and deep-learning-based
computational methods have been universally recognized as potentially
capable of overcoming the curse of
dimensionality~\cite{grohs2018proof, poggio2017and}.  Depending on how
the loss function is formulated, deep-learning-based methods for
deterministic differential equations can broadly be classified into
three categories: 1) residual based minimization, 2) deep backward
stochastic differential equations (BSDE) and 3) energy based
minimization.  Both the physics informed neural networks
(PINNs)~\cite{raissi2019physics, raissi2017physics} and the deep
Galerkin methods (DGMs)~\cite{sirignano2018dgm} belong to the category
of residual minimization.  PINNs aim to minimize the total residual
loss based on a set of discrete data points in the domain $D$.
Similarly, DGMs rely on minimizing the loss function induced by the
$L^2$ error.  We also refer~\cite{berg2018unified, carleo2017solving}
for a similar approach.  In contrast, the deep-BSDE methods exploit
the intrinsic connection between PDEs and BSDEs in the form of the
nonlinear Feynman-Kac formula to recast a PDE as a stochastic control
problem that is solved by reinforcement learning
techniques~\cite{han2017deep,han2018solving, beck2019machine,
  han2016deep, nusken2021solving}.  See~\cite{nusken2021interpolating}
for a recent work on interpolating between PINNs and deep BSDEs.
Finally, energy methods construct the loss function by taking
advantage of the variational formulation of elliptic
PDEs~\cite{yu2018deep, khoo2019solving, khoo2021solving,
  muller2019deep}.  As a comparable approach,
in~\cite{li2022semigroup} the Dynkin formula for overdamped Langevin
dynamics is employed to construct a weighted variational loss
function. We refer an interested reader to~\cite{beck2020overview,
  weinan2021algorithms, weinan2020machine} for an in-depth overview of
deep-learning-based methods for solving PDEs.

Lately, extensions of the deep-learning approaches for PDEs with
random coefficients have been proposed.  For example, convolutional
neural networks mimicking image-to-image regression in computer vision
have been employed to learn the input-to-output mapping for parametric
PDEs~\cite{khoo2021solving,zhu2019physics, zhang2021bayesian}.
However, these methods, by design, are mesh-dependent and therefore
aim to discover a mapping between finite dimensional spaces.  In
contrast, operator learning methods, such as the deep operator network
(ONET)~\cite{lu2019deeponet} and the Fourier neural operator
(FNO)~\cite{li2020fourier}, aim at learning a mesh-free, infinite
dimensional operator with specialized network architectures.  Let us also point out that all
these methods tend to identify a mapping $u: \kappa \to u(\kappa)$,
i.e., from the random field to the solution.  In practice, the
sampling of $\kappa$ is often reliant on the Karhunen–Lo{\`e}ve
expansion to parameterize $\kappa$ by means of a finite dimensional
random vector $Z$.  Taking this fact into account, in this article we
propose to directly learn a mapping $u: (x, Z) \to u(x, Z)$, i.e.,
from the joint space of $x$ and $Z$ to the solution.  
We empirically show that our approach achieves comparable accuracy with traditional methods even with a generic fully connected neural network architecture.
More importantly, however,
ONET and FNO are supervised by design and rely on the mean square loss
function to facilitate training.  Such a loss function rarely takes
the underlying physics into account in a
direct manner and customarily depends on sampling of the solution
throughout the domain.
In order to sidestep the above difficulties, we instead
employ~\eqref{eqn:abstract-energy} directly.  We emphasize that our
approach possesses a fundamental advantage over the approaches reliant
on stochastic differential equations.  Namely, after approximating the
solution $u$ by a DNN, the functional provides a natural loss
function. Hence, the physical laws encoded
in~\eqref{eqn:abstract-energy} are seamlessly incorporated when a
minimizer is computed.  Our method can be considered as an extension
of the deep-Ritz method in~\cite{yu2018deep} to the stochastic setting
and henceforth we refer to it as the stochastic deep-Ritz method.

The remainder of the article is organized as follows.
Section~\ref{sec:model} describes the general framework for solving
the stochastic variational problem by means of DNN. In
Section~\ref{sec:stochastic-deep-Ritz}, we propose the DNN based
approach and present the stochastic deep-Ritz method.  Finally,
numerical benchmarks are presented in Section~\ref{sec:example}.
 
\section{The stochastic variational problem}\label{sec:model}

Let $D \subset \mathbb{R}^d$ be a bounded open set with a Lipschitz
boundary $\partial D$ and $(\Omega, \mathcal{F}, \mathbb{P})$ be a probability
space, where $\Omega$ is the sample space,
$\mathcal{F} \subset 2^{\Omega}$ is a $\sigma$-algebra of all possible
events and $\mathbb{P}: \mathcal{F} \to [0, 1]$ is a probability
measure.  We consider functionals of the type 
\begin{equation}
\label{eq:sochastic_functional}
   J(u) = \mathbb{E}\left[\int_D I(x, u, \nabla u; \kappa(x,  \omega))\, dx\right]
\end{equation}
where $I:\R^d \times \R^N \times \R^{d\times N} \to \R$ is the Lagrangian,
$u : \bar{D} \times \Omega \to \R^N$ is the solution of interest,  and
$\kappa: \bar{D} \times \Omega \to \R$ is the random input field.  
Here $\bar{D} = D \cup \partial D$
denotes the closure of $D$ in $\R^d$ and $\nabla$ is the gradient with
respect to the spatial coordinate $x$. We suppose that the Lagrangian
$I$ is of class $C^1$ on an open subset of
$\R^d \times \R^N \times \R^{d \times N}$.  Moreover, we make the usual
assumption that the random fields $\kappa$ is parameterized by a
finite dimensional random vector
$Z = (Z_1, \ldots, Z_K): \Omega \to \Gamma \subset \mathbb{R}^K$,
i.e., there exists finitely many uncorrelated random variables
$Z_1, \ldots, Z_K$ such that the random field is of the form
\[\kappa(x, \omega) = \kappa(x, Z(\omega))\]
and hence~\eqref{eq:sochastic_functional} reduces to 
\begin{equation}
\label{eq:sochastic_functional_finite_dim}
   J(u) = \mathbb{E}_Z\left[\int_D I(x, u, \nabla u; \kappa(x,  Z))\, dx\right],
\end{equation}
where the expectation is taken with respect to the random vector $Z$.
We are interested in solving the following stochastic variation problems
\begin{equation}
\label{eqn:stochastic-var-form}
\min_{u \in \mathcal{U}} J(u)
\end{equation}
over a suitable functional space $\mathcal{U}$. There exists a
substantial body of literature dedicated to studying the existence of
solutions
of~\eqref{eqn:stochastic-var-form}~(see, for example, \cite{giaquinta2013calculus,dacorogna2007direct}). We
direct the reader to Theorem~\ref{thm:existence-and-uniqueness} in the appendix where we derive the conditions for
existence of such solutions for a quadratic Lagrangian.

Numerical solutions of~\eqref{eqn:stochastic-var-form} are notoriously
difficult to obtain. While many techniques have been proposed, they
primarily rely on, first, reducing the minimization
problem~\eqref{eqn:stochastic-var-form} to a random input PDE
and, second, numerically solving the
resulting random input PDE by classical techniques such as stochastic collocation
(SC)~\cite{babuvska2007stochastic, nobile2008sparse,
  nobile2008anisotropic}, stochastic Galerkin (SG)
\cite{ghanem2003stochastic, babuska2004galerkin,
  gunzburger2014stochastic, xiu2002wiener, xiu2003modeling} and Monte
Carlo (MC) \cite{babuska2004galerkin, matthies2005galerkin,
  kuo2016application}. In contrast to these techniques, we proposed
in~\cite{wang2021stochastic} an approach to seek solutions
of~\eqref{eqn:stochastic-var-form} directly by a combination of
polynomial chaos expansion (PCE) and stochastic gradient descent
(SGD). 
However,  the approach relies on the approximation over the tensor product space of $D$ and $\Omega$. 
In what follows,  we introduce a DNN-based method
which does not require such approximation over the tensor product space.

\section{Stochastic deep-Ritz: deep learning based UQ}\label{sec:stochastic-deep-Ritz}
\subsection{Preliminaries on neural networks}
Let us now briefly review some preliminaries on DNN and set up the
finite dimensional functional space induced by DNN.  A standard
fully-connected DNN of depth $L > 2$ with input dimension $N_0$ and
output dimension $N_L$ is determined by a tuple
\[
\Phi = \{(T_1, \sigma_1), \ldots, (T_L, \sigma_L)\},
\]
where
$T_l(x) = W_l x + b_l$ is an affine transformation with weight matrix $W_l \in \mathbb{R}^{N_{l-1} \times N_l}$ and bias vector $b_l \in \mathbb{R}^{N_l}$ and $\sigma_l : \mathbb{R}^{N_l} \to \mathbb{R}^{N_l}$ is a nonlinear activation function.
We say the DNN $\Phi$ is of width $W$ if $W =  \max_{l=1, \ldots, L} N_l$. 
We denote
\[\theta = (W_1, b_1, \ldots, W_L, b_L)\]
the weight parameters of $\Phi$
and 
\[
\sigma = (\sigma_1,  \ldots,  \sigma_L)
\]
the set of activation functions of $\Phi$.
For a fixed set of activation functions $\sigma$,  the DNN $\Phi$ is completely determined by the weight parameters $\theta$ and hence in the sequel
we write $\Phi = \Phi_{\theta}$ to emphasize the dependence on $\theta$. 
Each DNN $\Phi_{\theta}$ induces a function 
\[
u_{\theta} = R(\Phi_{\theta}) : \mathbb{R}^{N_0} \to  \mathbb{R}^{N_L},
\quad x \mapsto \sigma_L \circ T_L \circ \ldots \circ \sigma_1 \circ T_1.
\]
We refer to $R(\Phi_{\theta})$ as the realization of the DNN $\Phi_{\theta}$. 
The set of all DNNs of depth $L$ and width $W$,  denoted by
$\mathcal{N}_{L , W}$, induces a functional space $\mathcal{U}_{L, W}$
consisting of all realizations of DNNs contained in
 $\mathcal{N}_{L , W}$, i.e.,
\[
\mathcal{U}_{L , W} = \left\{ u_{\theta}: \mathbb{R}^{N_0} \to  \mathbb{R}^{N_L}~|~ u_{\theta} = R(\Phi_{\theta}) ,  \Phi_{\theta} \in \mathcal{N}_{L , W} \right\}.
\]
Approximating a function $u$ by DNNs in $\mathcal{N}_{L , W}$ amounts
to selecting a realization $u_{\theta} \in \mathcal{U}_{L , W}$
minimizing an appropriate loss functions.  Therefore, a DNN along with its
realization provide together a surrogate over the functional
space $\mathcal{U}_{L , W}$.

\subsection{The stochastic deep-Ritz method}
We are now in a position to discuss solving the stochastic
variational problem~\eqref{eqn:stochastic-var-form} with DNN. To
avoid the deterministic approximation of the integration over the
physical domain $D$, we introduce a $d$-dimensional random vector $X$ with the
uniform density function $\mu(x)$ supported on $D$ and
rewrite~\eqref{eqn:stochastic-var-form} into the following fully
probabilistic form:
\begin{equation}\label{eqn:fully-prob-form}
\min_{u \in \mathcal{U}} J(u) \quad \textrm{with}\quad
J(u) =
\mathbb{E}\left[I(X, u(X,Z), \nabla u(X,Z); \kappa(X, Z)) \right],
\end{equation}
where the expectation $\mathbb{E}$ is now taken with respect to both
$X$ and $Z$ and we implicitly assume that the volume of $D$ with
respect to $\mu$ is normalized to $1$, i.e., $\int_D \mu(x)\, dx = 1$.

Next, we seek an approximate solution of the
problem~\eqref{eqn:fully-prob-form} in the neural network space
$\mathcal{U}_{L , W}$.  The functional in~\eqref{eqn:fully-prob-form}
provides a natural loss function that can be optimized with DNN using
SGD methods.  Specifically, we consider a class of DNNs
$\Phi_{\theta}$ (parameterized by $\theta$) with input dimension
$d + K$ and output dimension $1$ whose realizations approximate the
true solution $u(x, z)$, i.e.,
\[
u(x,  z)  \approx u_{\theta}(x,  z) = R(\Phi_{\theta})(x,  z),
\]
through minimizing the functional~\eqref{eqn:fully-prob-form}.
This leads to the following finite dimensional optimization problem:
\[
\min_{u_\theta \in \mathcal{U}_{L, W}} J(u_\theta) \quad \textrm{with}\quad
J(u_\theta) =
\mathbb{E}\left[ I(X, u_\theta(X,Z), \nabla u_\theta(X,Z); \kappa(X, Z)) \right],
\]
where the optimization is over the space $\mathcal{U}_{L, W}$, i.e.,
the function space induced by the realizations of DNNs with depth $L$
and width $W$.  The fully-connected DNN architecture that we employ in
this work is schematically represented in Figure~\ref{fig:NN}.
Although there exist more sophisticated NN architectures, e.g.,
residual NN~\cite{he2016deep} and Unet~\cite{ronneberger2015u},  we demonstrate in
Section~\ref{sec:example} that a rather simple fully-connected DNN is
capable of learning the solution with high accuracy.

Unlike the finite element method, which takes the boundary conditions
into account by design,  hard-constraint boundary conditions 
for DNN-based methods can only be enforced for simple geometry of $D$~\cite{lu2021physics, sun2020surrogate, muller2021notes}. 
Therefore,  we enforce the soft-constraint boundary condition by introducing a penalty term to compensate for the boundary condition and
obtain the final loss function:
\begin{equation}\label{eqn:stochastic-optim-problem}
\min_{u_{\theta} \in \mathcal{U}_{L, W}}J(u_{\theta}) 
\end{equation}
with 
\[
  J(u_{\theta})
  = \mathbb{E}\left[ I(X, u_\theta(X,Z), \nabla u_\theta(X,Z); \kappa(X, Z)) \right]
  + \beta  \mathbb{E}\left[ |u_{\theta}(S, Z)|^2\right],
\]
where $\beta$ is the penalty coefficient and $S$ is a random vector
with a uniform distribution over the boundary $\partial D$.  Here, the
first and second expectations are taken with respect to $(X, Z)$ and
$(S, Z)$, respectively. We use the same $\mathbb{E}$ to denote
the two expectations for notational simplicity.  
It should be emphasized that, other than the DNN approximation error,  the penalty term often introduces an additional source of errors
into the problem and hence changes the ground truth solution.
However,  we expect that the soft constraint error is negligible compared to the DNN approximation error for large penalty $\beta$.  
The stochastic
optimization problem~\eqref{eqn:stochastic-optim-problem} can be naturally solved by the SGD algorithm or its variants and the final
stochastic deep-Ritz algorithm is presented in
Algorithm~\ref{alg:stochastic-deep-Ritz}.

\begin{figure}
\centering
\includegraphics[scale=1.0]{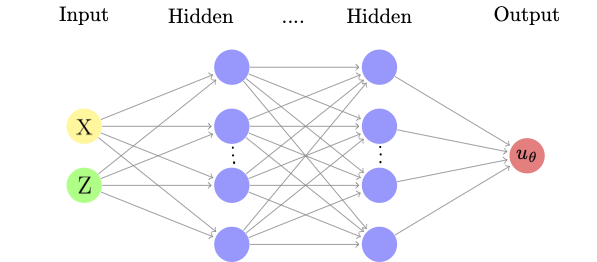}
\caption{The fully-connected neural network architecture employed in
  the stochastic deep-Ritz solver.}
\label{fig:NN}
\end{figure}

\begin{algorithm}[!ht]
\caption{The stochastic deep-Ritz algorithm}\label{alg:stochastic-deep-Ritz}
\textbf{Input}: A fully-connected DNN $\Phi_{\theta}$; number of
training iterations $N_e$; mini-batch size $M$; penalty coefficient $\beta$;
learning rate $\eta_n$;\\
\textbf{Output}: A DNN surrogate $\Phi_{\theta}$ with realization
$u_{\theta}$.
\begin{algorithmic}[1]
\For{$n = 1 : N_e$}
  \State Generate a minibatch of samples $(X_m,  S_m,  Z_m)$ for $m = 1, \ldots, M$,
  \State Evaluate the NN realizations $u_{\theta}(X_m, Z_m)$ and $u_{\theta}(S_m, Z_m)$,
  \State Compute the stochastic gradient estimator $g(\theta_n; X_m, S_m, Z_m)$ such that
  \[\mathbb{E}\left[   g(\theta_n; X_m, S_m, Z_m)   \right] = \nabla_{\theta} J(\theta_n),  \quad m = 1, \ldots, M, \]
  \State Compute the mini-batch average: 
  \[\bar{g}(\theta_n) = \frac{1}{M} \sum_{m=1}^{M}  g(\theta_n; X_m, S_m, Z_m), \]
  \State Update the DNN parameters: $\theta_n \gets \theta_n - \eta_n \bar{g}(\theta_n)$.
\EndFor
\end{algorithmic}
\end{algorithm}
\begin{remark} Some remarks on
  Algorithm~\ref{alg:stochastic-deep-Ritz} are in order.
\begin{enumerate}
\item A natural choice for the stochastic gradient estimator is 
  \[
    g(\theta; X, S, Z) =
    \nabla_{\theta} \left\{ I(X, u_\theta(X,Z), \nabla u_\theta(X,Z); \kappa(X, Z))
      + \beta \left|u_{\theta}(S, Z)\right|^2\right\},
\]
which can be easily obtained by automatic differentiation in any deep
learning frameworks.  However, it is possible to construct more
efficient estimators through variance reduction
techniques~\cite{asmussen2007stochastic, glasserman2004monte}. This
will be the focus of another work.
\item The numerical consistency of the stochastic deep-Ritz method,
  i.e., the minimizer of $J(u_{\theta})$ over $\mathcal{U}_{L, W}$
  converges (weakly) to the minimizer of $J(u)$ over $\mathcal{U}$,
  relies on justifying the $\Gamma$-convergence when either the depth
  $L \to \infty$ or the width $W \to \infty$ (see e.g.,
  ~\cite{struwe1990variational}).  We refer interested reader
  to~\cite{muller2019deep} for a proof in the deterministic setting.
\item In practice, more efficient SGD variants such as Adam can be
  used to expedite the training of NNs~\cite{kingma2014adam}.
\end{enumerate}
\end{remark}

It should be emphasized that the stochastic deep-Ritz method learns a
deterministic function $u_{\theta}$ mapping the random variables $X$
and $Z$ to the random variable $u_{\theta}(X, Z)$.  Similar
applications of DNNs are common in deep learning.  For instance, the
generator of a GAN learns a mapping from Gaussian samples to samples
from the desired distribution.  Similarly,  the DNN in our method takes
samples from $(X, Z)$ and learns how to generate samples from
$u_{\theta}(X, Z)$, that is, the DNN learns the transformation between
the input $(X, Z)$ and solution $u_{\theta}(X, Z)$.  In comparison with the traditional
approaches, e.g., SC and PC, which obtain an explicit approximation to
the solution, the proposed methodology should be considered as an
implicit method since it only allows us to learn an approximated
simulator of the true solution $u(X, Z)$.  However,  note that for stochastic problems one is often not interested in the analytical solutions but rather the statistical properties (e.g., moments, distribution) of the solution. 
In practice, all statistical properties of the solution can be easily
computed through sampling directly from the learned neural DNN.

\section{Numerical experiments}\label{sec:example}
In order to demonstrate the efficacy and accuracy
of the proposed method,  we assess its numerical performance by testing it on a class of stochastic variational problems with quadratic Lagrangian, i.e., 
\begin{equation}\label{eqn:Ritz-form}
\min_{u \in \mathcal{U}}  J(u) \quad \textrm{with}\quad J(u) =  \mathbb{E}\left[\int_D \frac{1}{2}\kappa(x, Z) |\nabla u(x, Z)|^2  - f(x, Z) u(x, Z)\, dx\right].
\end{equation}
Note that the Euler-Lagrange form of above problem corresponds to the following class of random elliptic PDEs 
\begin{equation}\label{eqn:PDE-finite-noise}
  \begin{aligned}
    -\nabla \cdot \left(\kappa(x, Z(\omega)) \nabla u(x,  Z(\omega))\right) &= f(x, Z(\omega)), & \text{in~} & D,\\
    u(x, Z(\omega)) &= 0, & \text{on~} & \partial D.
  \end{aligned}
\end{equation}
For completeness,  we provide in the appendix the theoretical justification that~\eqref{eqn:Ritz-form} admits a unique minimizer which satisfies the stochastic  weak form~\eqref{eqn:weak-form} of~\eqref{eqn:PDE-finite-noise}.

We measure the accuracy of the DNN realization
$u_{\theta}(x, Z)$ using the relative $L^2$ mean error
\[
  E(\theta) = \frac{\mathbb{E}[ ||u_{\theta}(X, Z) - u(X, Z)||^2
    ]}{\mathbb{E}[ ||u(X, Z)||^2]}.
\]
For all the numerical examples considered in this section, we use the
fully-connected DNN architecture (see~Figure~\ref{fig:NN}) with the
$\tanh$ activation function and apply the Adam
optimizer~\cite{kingma2014adam} to train the DNN.  
More specialized DNN architectures for solving the stochastic variational problem will be the focus of another work.  
All numerical
experiments are implemented with the deep learning framework
PyTorch~\cite{paszke2017automatic} on a Tesla V100-SXM2-16GB GPU.

\subsection{A one-dimensional problem}\label{sec:1d_problem}
In the first example, we test the performance of the stochastic
deep-Ritz solver by considering the one-dimensional Lagrangian ($d=N=1$)
\begin{equation}
  \label{eqn:1d-Lagrangian}
  I(x, u, u'; \kappa) = \frac{1}{2} \kappa(x,Z) (u'(x,Z))^2.
\end{equation}
We take $D$ to be the open interval $(-1,1)$ in $\R$.  The boundary
conditions for $u$ on $\partial D = \{-1, +1\}$ are $u(-1, Z) = 0$ and
$u(+1, Z) = 1$.  The random field $\kappa(x, Z)$ is taken to be
log-normal of the form $\kappa = \mathrm{e}^{\beta V(x, Z)}$ with
$\beta = 0.1$, where the potential $V$ is a nonlinear function of the
random vector $Z = (A_1, \ldots, A_n, B_1, \ldots, B_n)$, namely,
$$V(x, Z) = \frac{1}{\sqrt{n}} \sum_{k=1}^{n} A_k \cos(\pi k x) + B_k \sin(\pi k x)$$ and $A_k$ and $B_k$
are independent unit-normal random variables. One can easily verify
that $V(x, Z)$ is a Gaussian random field with zero mean and the
covariance kernel
$$\text{cov}(x_1, x_2) = \frac{1}{n}\sum_{k=1}^{n} \cos(\pi k (x_2 - x_1)).$$
It can be shown that the exact solution of the
problem~\eqref{eqn:stochastic-var-form} with the
Lagrangian~\eqref{eqn:1d-Lagrangian} is
$$u(x, Z) = \left(\int_{-1}^{1} \frac{1}{\kappa(\xi, Z)}\, d\xi\right)^{-1} \int_{-1}^{x} \frac{1}{\kappa(\xi, Z)}\, d\xi$$

We test the algorithm for this problem with $n = 5$ (i.e., $10$ dimensional random vector).
We use a fully-connected DNN with $L = 5$
layers, $W = 256$ nodes per layer and tanh activation function. To
account for the boundary condition, we set the penalty coefficient to
be $\beta = 50$.  For the training of the DNN, we run the solver for
$4\times 10^5$ iterations with a minibatch size $M = 2560$ per iteration.  The
initial learning rate is set to $\eta = 10^{-3}$ and is reduced by a
factor of $10$ after every $10^5$ iterations.  At the end of the training,
the final relative $L^2$ mean error is $E(\theta) = 0.52\%$.  

We compare our method with the ONET developed in~\cite{lu2019deeponet}. 
In contrast, the ONET is a data driven method. 
We train the ONET based on a data set consists of $10^4$ training samples. 
The branch component of the ONET has $100$ units at the input layer and $20$ units at the output layer with $3$ $128$-unit hidden layers in the middle. 
The trunk component of the ONET has a single unit input layer,  $20$-unit output layer 
and $3$ hidden layers with $128$ units per layer. The choice of the minibatch size, number of iterations and learning rate are the same as that of the stochastic deep-Ritz. 


We compare the one-dimensional marginal distributions (i.e., $u(x, Z)$ at different $x$) of the stochastic deep-Ritz solution, ONET solution and the exact
solution in Figure~\ref{fig:1d_field}.  The results
demonstrate that our method approximates the true random field
solution fairly accurately.  In contrast, the ONET fails to capture the tail behaviors of the distribution.

\begin{figure}[h]
\centering
\includegraphics[width=1.0\linewidth]{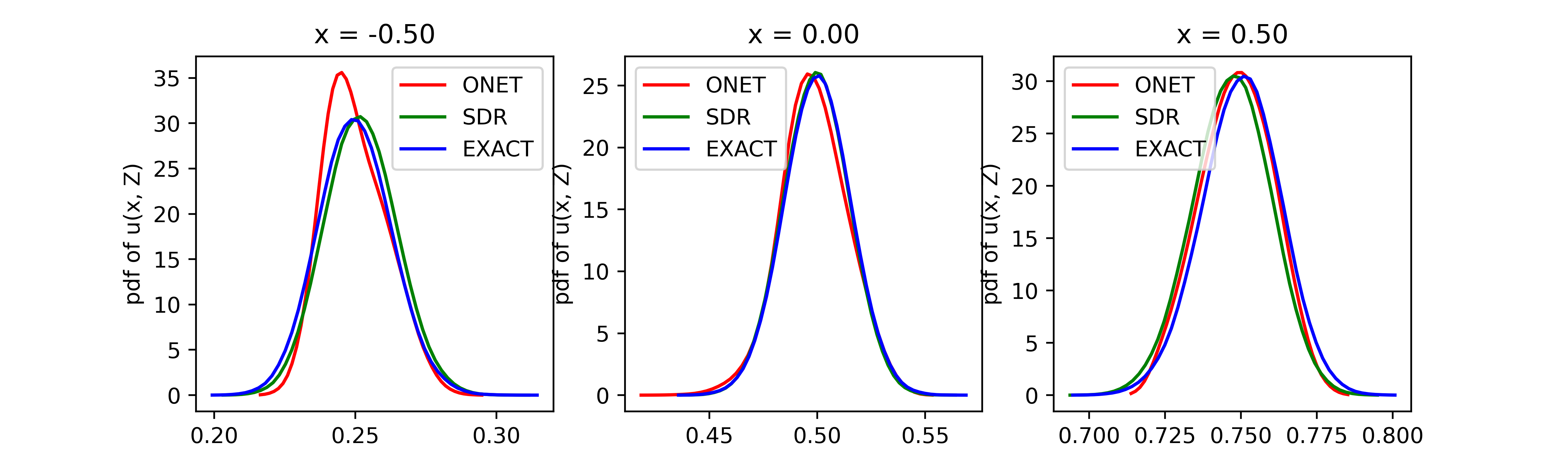}
\caption{Marginal distributions of the solution to the 1d Dirichlet problem in Section~\ref{sec:1d_problem}. Left: the pdf of $u_{\theta}(x, Z)$ at $x=-0.5$; Center: the pdf of $u_{\theta}(x, Z)$ at $x=0$; Right: the pdf of $u_{\theta}(x, Z)$ at $x=0.5$. }
\label{fig:1d_field}
\end{figure}

Finally, we visualize the joint density for $(u_{\theta}(-0.5, Z), u_{\theta}(0.5, Z))$ in Figure~\ref{fig:1d_field_joint_pdf}, which verifies that the approximated solution $u_{\theta}(x, Z)$ captures the correct correlation of the random field at different $x$.

\begin{figure}[ht!]
\centering
\includegraphics[width=0.9\linewidth]{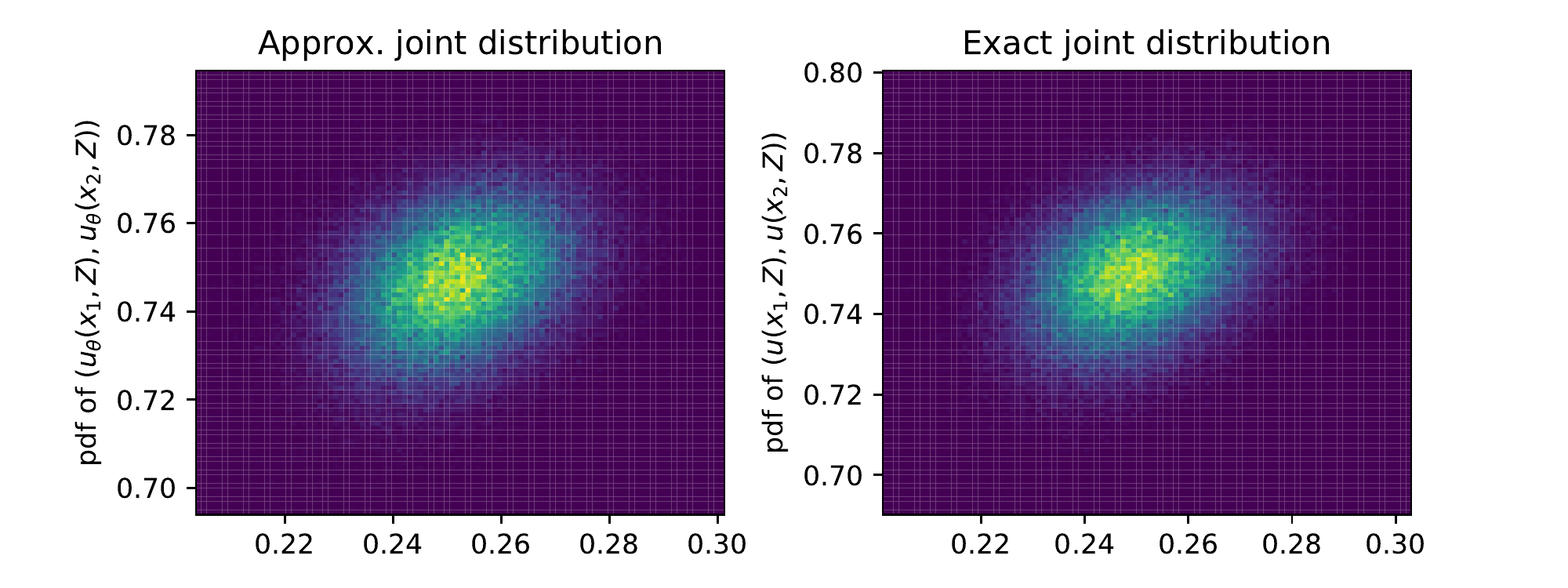}
\caption{Joint distribution of the solution to the 1d Dirichlet problem in Section~\ref{sec:1d_problem}. Left: the approximated joint pdf of $u_{\theta}(-0.5, Z)$ and $u_{\theta}(0.5, Z)$; Right: the exact joint pdf $u(-0.5, Z)$ and $u(0.5, Z)$.}
\label{fig:1d_field_joint_pdf}
\end{figure}

\subsection{A quadratic Lagrangian with Neumann boundary condition}\label{sec:Neumann_BC}
Next, we consider the following Lagrangian
\begin{equation}
  \label{eqn:semilinear_Lagrangian}
  I(x,u,\nabla u; \kappa) = \frac{1}{2} \kappa(x, Z) |\nabla u(x, Z)|^2 - f(x, u(x, Z), Z)\ u(x, Z)
\end{equation}
While we take $d$ arbitrary for now, $N = 1$ is assumed. The physical
domain is $D = [0, 1]^d$. We impose the following Neumann boundary
conditions on $\partial D$
\[
\frac{\partial u}{\partial n}(x, Z) = 0.
\]
The random field $\kappa$ is taken to be
\[\kappa(x, Z) = d+ 1 + \sum_{i=1}^d Z_i\]
with $Z = (Z_1, \ldots, Z_d)$ uniformly distributed over $D$ and the
semilinear term $f$ is given by
\[
  f(x, u(x, Z), Z) = -\pi^2 (d+ 1 + \sum_{i=1}^d Z_i)\ u(x, Z) +
  2\pi^2 \sum_{i=1}^d\cos(\pi x_i).
\]
It can be easily verified that the unique solution of the minimization
problem~\eqref{eqn:stochastic-var-form} with the
Lagrangian~\eqref{eqn:semilinear_Lagrangian} is
$$
u(x, Z) = \frac{1}{\kappa(x, Z)} \sum_{i=1}^d\cos(\pi x_i).
$$
To solve the above minimization problem numerically, we approximate
the solution $u(X, Z)$ by a DNN realization $u_{\theta}(X, Z)$ and
formulate the following variational problem:
\[
\min_{\theta}J(\theta) \quad \textrm{with} \quad J(\theta) = \mathbb{E}\left[\frac{1}{2}\left( \kappa(X, Z) |\nabla u_{\theta}(X, Z)|^2 - f(X, u_{\theta}(X, Z), Z) u_{\theta}(X, Z)\right)
\right].
\]
Note that this problem does not require a penalty term and hence the ground truth remains unchanged.

In the numerical experiment with $d=2$, we use a $L = 5$ layer
fully-connected DNN with $W = 32$ nodes in each layer and $\tanh$
activation. We train the DNN for $3\times 10^5$ iterations with a
minibatch size $2560$ per iteration. The initial learning rate is set to
$\eta = 10^{-4}$ and is decayed by a factor of $0.1$ after every
$10^5$ iterations.  We test the accuracy of the trained DNN solution by
computing the the relative $L^2$ mean error using a Monte Carlo
integration.  Based on $10^4$ test samples, the final relative $L^2$
mean error is $E(\theta) = 0.65\%$.  The decays of both the loss
function $J(\theta)$ and the relative $L^2$ mean error $E(\theta)$
with respect to the number of iterations are visualized in
Figure~\ref{fig:Neumann-convergence}, which demonstrates a consistent
convergence behavior of $J(\theta)$ and $E(\theta)$.  These results
confirm that minimizing the loss $J(\theta)$ leads to a good
approximation to the true solution.

\begin{figure}[ht!]
\centering
\includegraphics[width=0.8\linewidth]{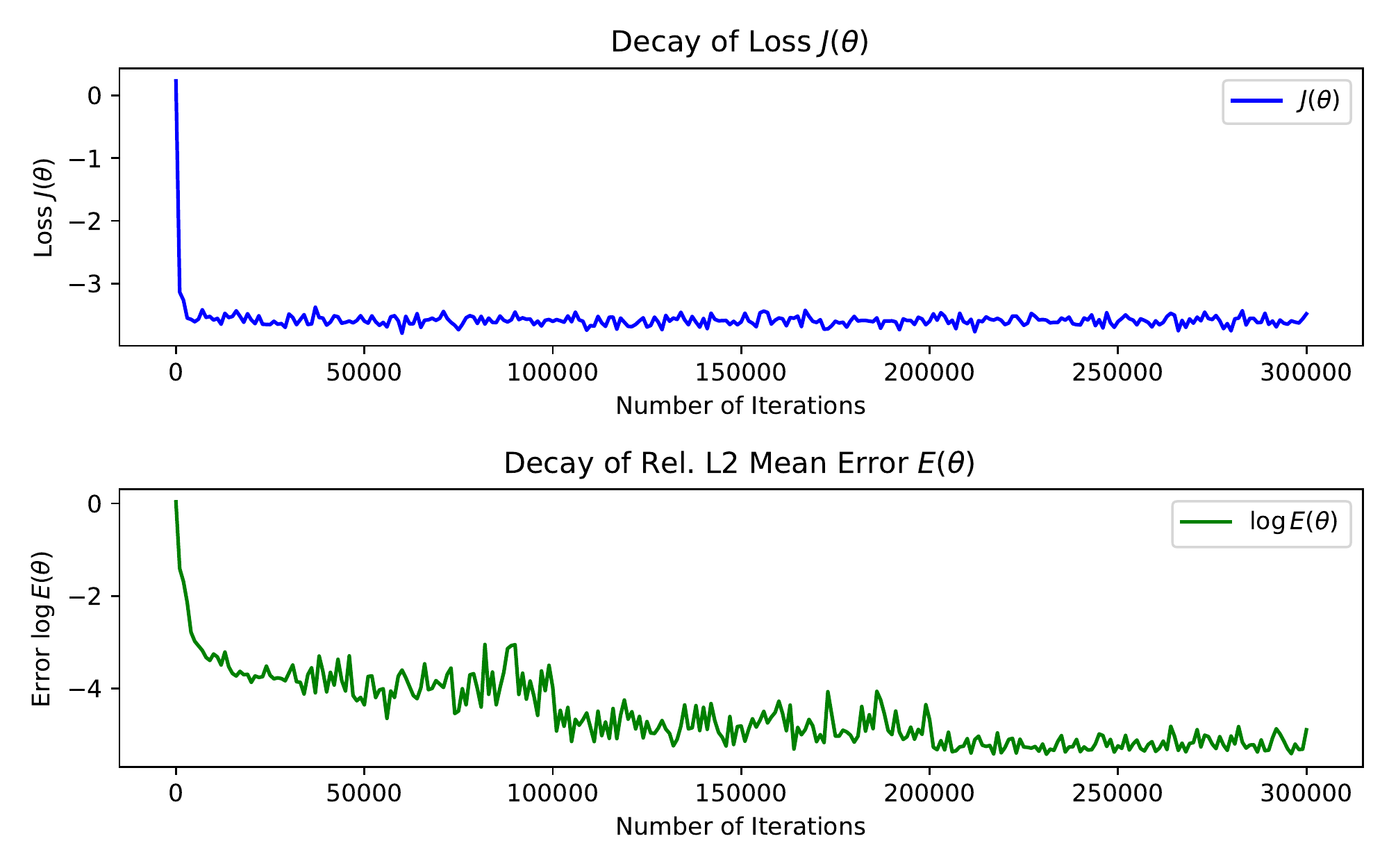}
\caption{Convergence of the stochastic deep-Ritz applied to the Neumann problem in Section~\ref{sec:Neumann_BC}.  {\bf Top}: decay of the variational loss.
 {\bf Bottom}: decay of the relative $L^2$ mean error (in log scale).}
\label{fig:Neumann-convergence}
\end{figure}

To demonstrate the distributional behavior of the DNN solution
$u_{\theta}(X, Z)$, we compare its distribution to the distribution of
the true solution $u(X, Z)$ at various values of $X$.  The probability
density functions (PDF) are shown in
Figure~\ref{fig:Neumann-nonlinear}. These results clearly indicate
that the stochastic deep-Ritz solver is capable of capturing the
distributional property of the true random field.

\begin{figure}[ht!]
\centering
\includegraphics[width=1.0\linewidth]{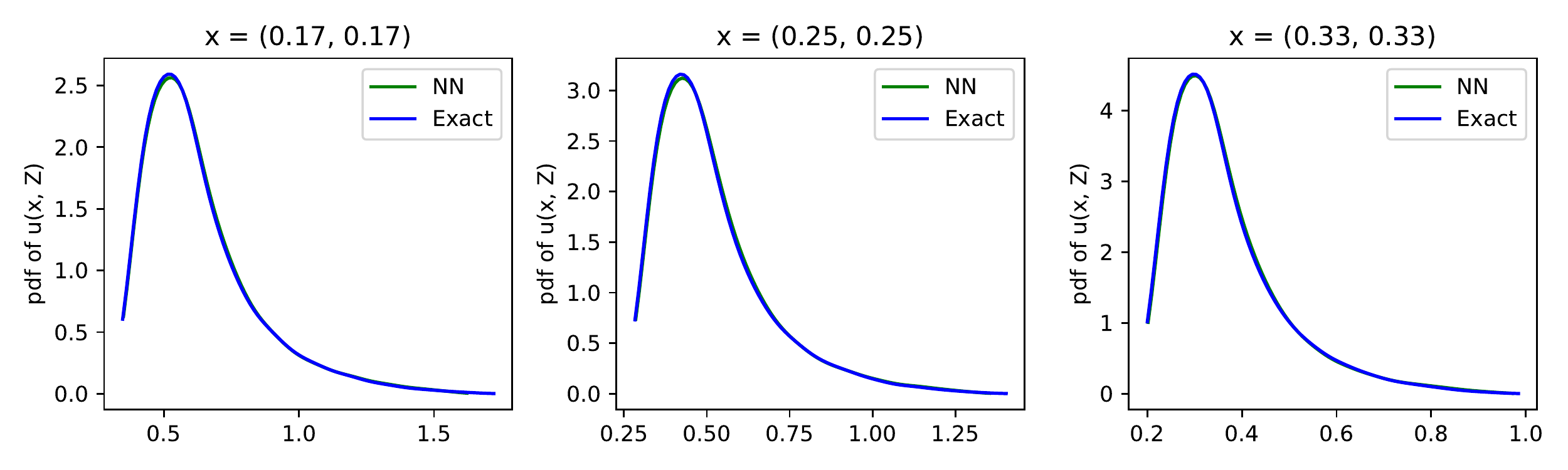}
\caption{Marginal distributions of the solution to the $2$ dimensional Neumann problem in Section~\ref{sec:Neumann_BC}. Left: the pdf of $u_{\theta}(x, Z)$ at $x=(1/6, 1/6)$; Center: the pdf of $u_{\theta}(x, Z)$ at $x=(1/4, 1/4)$; Right: the pdf of $u_{\theta}(x, Z)$ at $x=(1/3, 1/3)$.}
\label{fig:Neumann-nonlinear}
\end{figure}

We conclude this subsection by applying the stochastic deep-Ritz
method to the Neumann problem~\eqref{eqn:semilinear_Lagrangian} in $d=10$
dimensions.  We employ the same DNN architecture, number of iterations and minibatch size as that for the case $d=2$.  We set the
initial learning rate to $\eta = 10^{-5}$ and decay it after every
$10^5$ iterations by a factor of 10. The final relative
$L^2$ mean error $E(\theta)$ turns out to be $0.93\%$ and the
numerical results for the marginal distributions
$u_{\theta}(\cdot, Z)$ are presented in
Figure~\ref{fig:10D-Neumann-nonlinear}.  The results confirm that,
with the same computational cost, the stochastic deep-Ritz method
maintains the desired accuracy for high dimensional problems, a clear
indication that the method may have the potential to overcome the
curse of dimensionality.

\begin{figure}[ht!]
\centering
\includegraphics[width=1.0\linewidth]{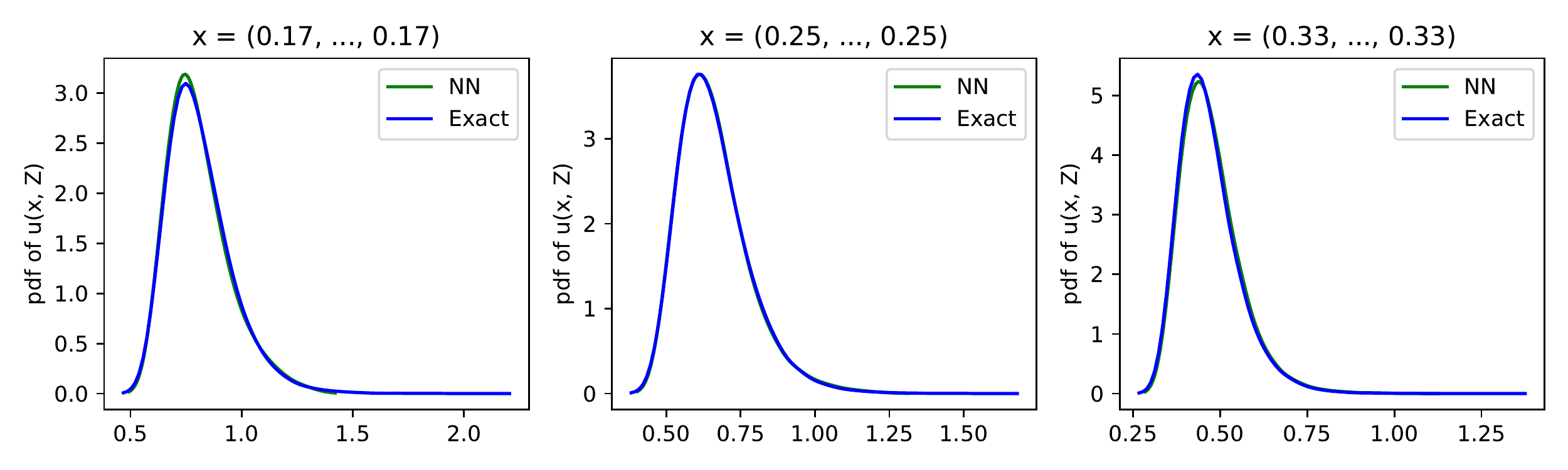}
\caption{Marginal distributions of the solution to the $10$ dimensional Neumann problem in Section~\ref{sec:Neumann_BC}.  Left: the pdf of $u_{\theta}(x, Z)$ at $x=(1/6, \ldots, 1/6) \in \mathbb{R}^{10}$; Center: the pdf of $u_{\theta}(x, Z)$ at $x=(1/4, \ldots, 1/4) \in \mathbb{R}^{10}$; Right: the pdf of $u_{\theta}(x, Z)$ at $x=(1/3, \ldots, 1/3) \in \mathbb{R}^{10}$.}
\label{fig:10D-Neumann-nonlinear}
\end{figure}

\subsection{A quadratic Lagrangian with Dirichlet boundary condition}\label{sec:quadratic_Lagrangian}
In this numerical example, we employ the
Lagrangian~\eqref{eqn:semilinear_Lagrangian} with $d=2$, $N=1$ and
\[
  f(x, u(x, Z), Z) = 2\pi^2 \sin(\pi x_1) \sin(\pi x_2).
\]
In contrast to the previous example, we take a zero boundary condition
$u(x, Z) = 0$ on $\partial D$. Here, the random field is defined as
\[\kappa(x, Z) = 3 + Z_1 + Z_2\]
with $Z = (Z_1, Z_2)$ uniformly distributed over $[-1, 1]^2$.
The problem admits the unique solution 
\[
u(x, Z) = \frac{1}{\kappa(x, Z)} \sin(\pi x_1) \sin(\pi x_2).
\]
In order to solve the problem using DNN, we seek a minimizer of the
loss function of the form~\eqref{eqn:stochastic-optim-problem}.

The numerical experiment is carried out with a $L = 5$ layer
fully-connected DNN with $W = 256$ nodes per layer and the tanh
activation function. The DNN is trained for $4 \times 10^5$ iterations
using $2560$ samples per minibatch for each iteration  We use a learning
rate scheduling with the initial learning rate $\eta = 10^{-3}$ and
decay the learning rate by a factor of $10$ after every $10^5$ iterations.
We plot the loss function $J(\theta)$ and the relative $L^2$ mean
error $E(\theta)$ in Figure~\ref{fig:Dirichlet-convergence_plot},
which again demonstrates that $J(\theta)$ and $E(\theta)$ decay in a
consistent manner.  Note that there is a clear decrease of $E(\theta)$
after the first $10^5$ iterations due to the learning rate scheduling.

\begin{figure}[htb]
\centering
\includegraphics[width=0.7\linewidth]{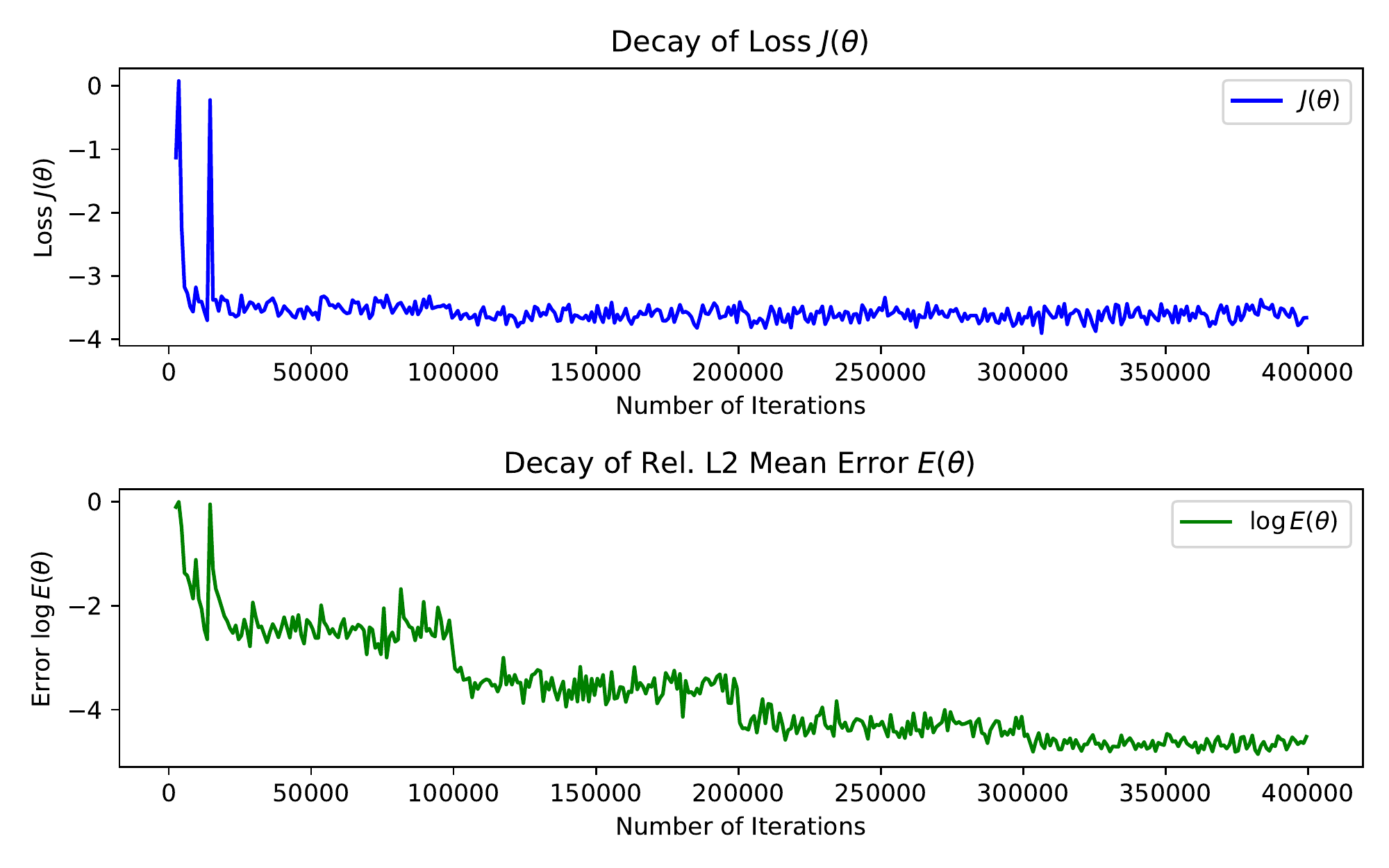}
\caption{Convergence of the stochastic deep-Ritz applied to the Dirichlet problem. The first $1000$ iterations is skipped in the plot. {\bf Top}: decay of the variational loss.
 {\bf Bottom}: decay of the relative $L^2$ mean error (in log scale).}
\label{fig:Dirichlet-convergence_plot}
\end{figure}

We gauge the accuracy of the DNN solution $u_{\theta}(x, Z)$ by
comparing it to the exact solution $u(x, Z)$.  With $10^5$ Monte Carlo
test samples, the relative $L^2$ mean error is $E(\theta) = 0.96\%$.
To further assess the accuracy of the approximated solution
$u_{\theta}(x, Z)$, we plot the marginal distribution and the joint
distribution of the random field in
Figure~\ref{fig:Dirichlet-nonlinear} and
Figure~\ref{fig:Dirichlet-joint-pdf}, respectively. The numerical
results suggest that the stochastic deep-Ritz method achieves a highly
accurate approximation to the true solution.


\begin{figure}[htb]
\centering
\includegraphics[width=1.0\linewidth]{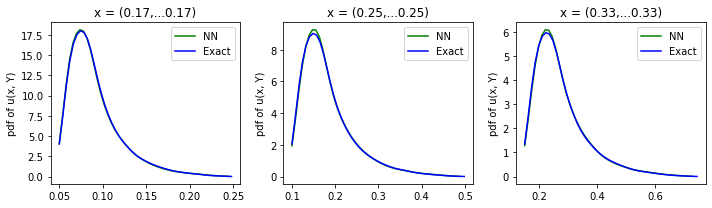}
\caption{Distributions of the solution to the Dirichlet problem in Section~\ref{sec:quadratic_Lagrangian} with penalty coefficient $\beta = 500$.  Left: the (unnormalized) pdf of $u_{\theta}(x, Z)$ at $x=(1/6, 1/6)$; Center: the (unnormalized) pdf of $u_{\theta}(x, Z)$ at $x=(1/4, 1/4)$; Right: the (unnormalized) pdf of $u_{\theta}(x, Z)$ at $x=(1/3, 1/3)$.}
\label{fig:Dirichlet-nonlinear}
\end{figure}

\begin{figure}[ht!]
\centering
\includegraphics[width=0.9\linewidth]{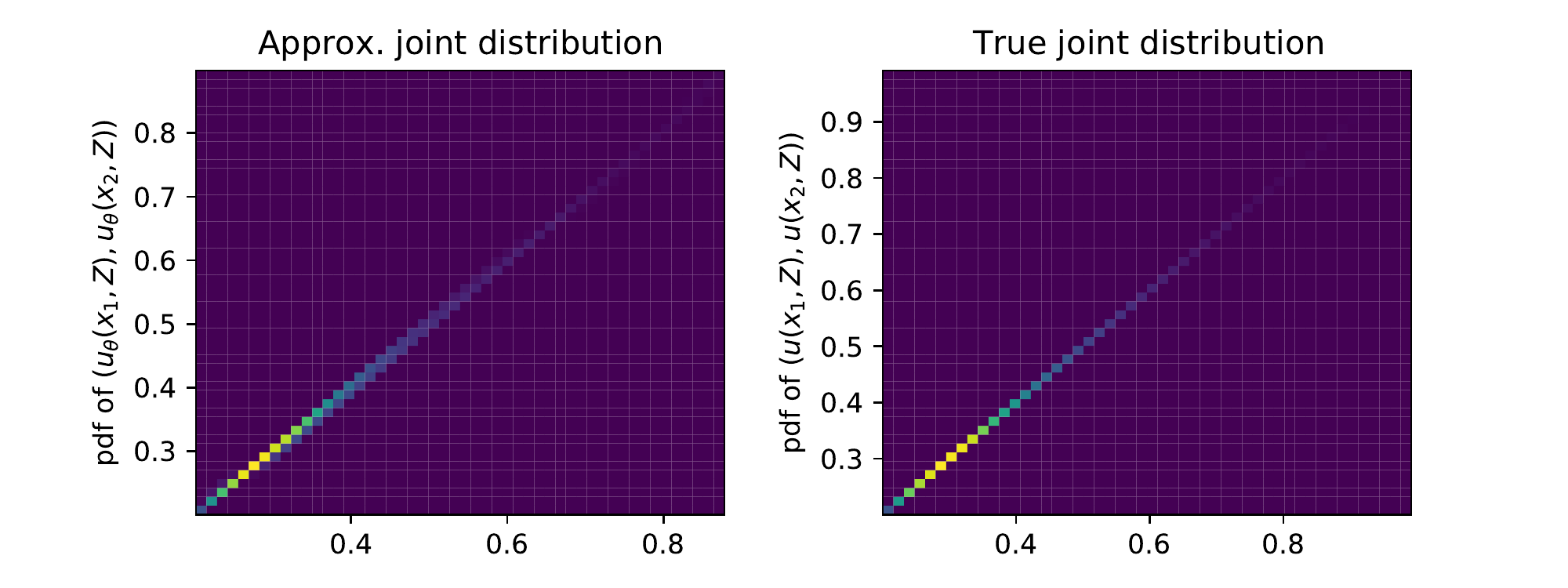}
\caption{Joint distribution of the solution to the Dirichlet problem in Section~\ref{sec:quadratic_Lagrangian} at $x_1=(-0.5, -0.5)$ and $x_2 = (0.5, 0.5)$. Left: the approximated joint pdf of $u_{\theta}(x_1, Z)$ and $u_{\theta}(x_2, Z)$; Right: the exact joint pdf $u(x_1, Z)$ and $u(x_2, Z)$.
}
\label{fig:Dirichlet-joint-pdf}
\end{figure}



%

\subsection{Overdamped Langevin dynamics with Dirichlet boundary condition}\label{sec:overdamped}
As the final test problem, we consider the
Lagrangian~\eqref{eqn:semilinear_Lagrangian} with $N=1$,
$f(x, Z) = -dZ$, and $\kappa(x, Z) = \textrm{e}^{-V(x, Z)}$, where
\[
  V(x, Z) = Z(1 + \|x\|^2)
\]
We take $D$ to be the $d$-dimensional unit ball $B(0,1)$. In addition, 
$u(x, Z) = \textrm{e}^{Z}$ on $\partial D$ is assumed.
The unique solution to the problem can be readily verified to be $u(x, Z) = \textrm{e}^{V}$.


Stochastic minimization of the loss function $J(\theta)$ associated
with this problem involves uniform sampling in $B(0, 1)$, which can be
performed by standard rejection sampling.  However, for high
dimensional problems (e.g., $d = 10$), the rejection sampling becomes
infeasible due to the high rejection rate.  We instead utilize the
ball point picking algorithm to sample uniformly in
$B(0, 1)$~\cite{barthe2005probabilistic}.  For the boundary penalty
term this amounts to sampling a uniform random vector $S$ on the
sphere $\partial B(0, 1)$. We use the fact that
\[
S = \frac{(X_1,  \ldots,  X_d)}{\sqrt{X_1^2 + \ldots + X_d^2}}
\]
is uniform on $\partial B(0, 1)$ when $X_1, \ldots, X_d$ are
independent identically distributed standard normal random variables.

\begin{figure}[h]
\centering
\includegraphics[width=1.0\linewidth]{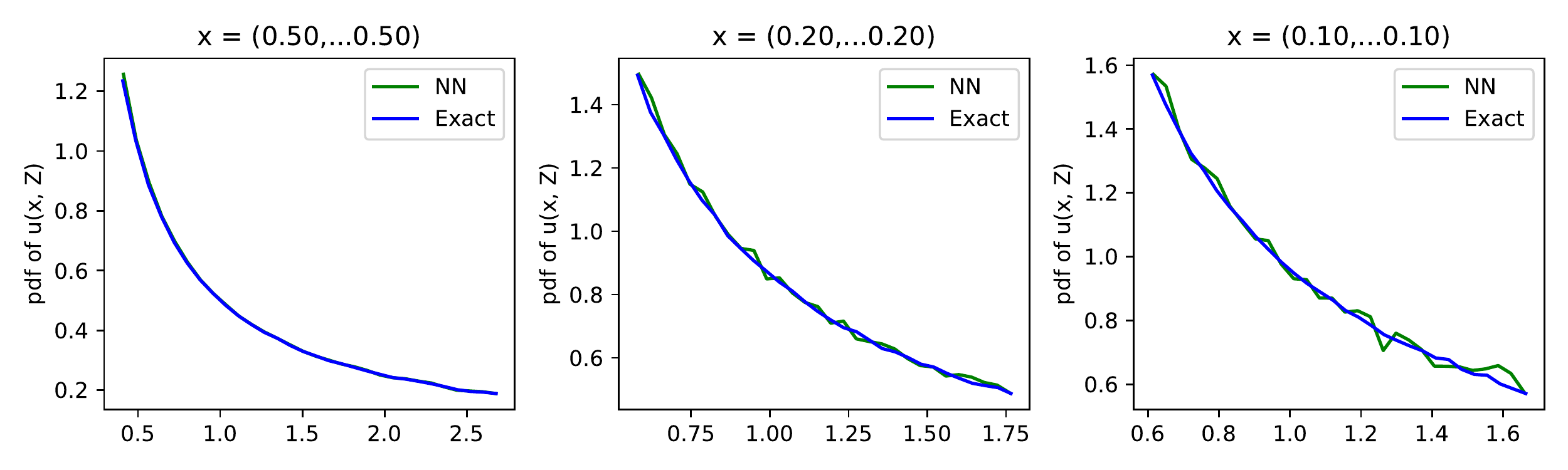}
\caption{Marginal distributions of the solution to the problem in Section~\ref{sec:overdamped} with $d = 4$
and penalty coefficient $\beta = 500$.  Left: the pdf of $u_{\theta}(x, Z)$ at $x = (0.5,  \ldots, 0.5) \in \mathbb{R}^4$; 
Center: the pdf of $u_{\theta}(x, Z)$ at $x = (0.2,  \ldots, 0.2) \in \mathbb{R}^4$; Right: the pdf of $u_{\theta}(x, Z)$ at $x = (0.1,  \ldots, 0.1) \in \mathbb{R}^4$.}
\label{fig:Langevin_4d}
\end{figure}


The numerical experiments are carried out with $d = 4$ and $d = 10$.
In both cases, we employ a $5$-layer fully-connected DNN with
$W = 256$ nodes per layer and we set mini-batch size to be $2560$.
The training of the DNN takes $3 \times 10^5$ iterations with an initial
learning rate $\eta = 10^{-3}$ that is decreased by a factor of $10$
after every $10^5$ iterations.  Finally, we choose the penalty coefficient
$\lambda = 500$ to enforce the boundary condition.  The numerical
results for $d = 4$ and for $d = 10$ are presented in
Figure~\ref{fig:Langevin_4d} and Figure~\ref{fig:Langevin_10d},
respectively, displaying the distribution of the solution
$u_{\theta}(x, Z)$ at different $x$.  The relative $L^2$ mean errors
are $E(\theta) = 0.19 \%$ for $d = 4$ and $E(\theta) = 0.43 \%$ for
$d = 10$.  The experiments, again, suggest that deep learning may
overcome the curse of dimensionality exhibited by traditional
discretization-based methods.
However, we also observe that the solution becomes noisy in the high dimensional setting since the optimization problem becomes more challenging when $d$ is large. 

\begin{figure}[ht!]
\centering
\includegraphics[width=1.0\linewidth]{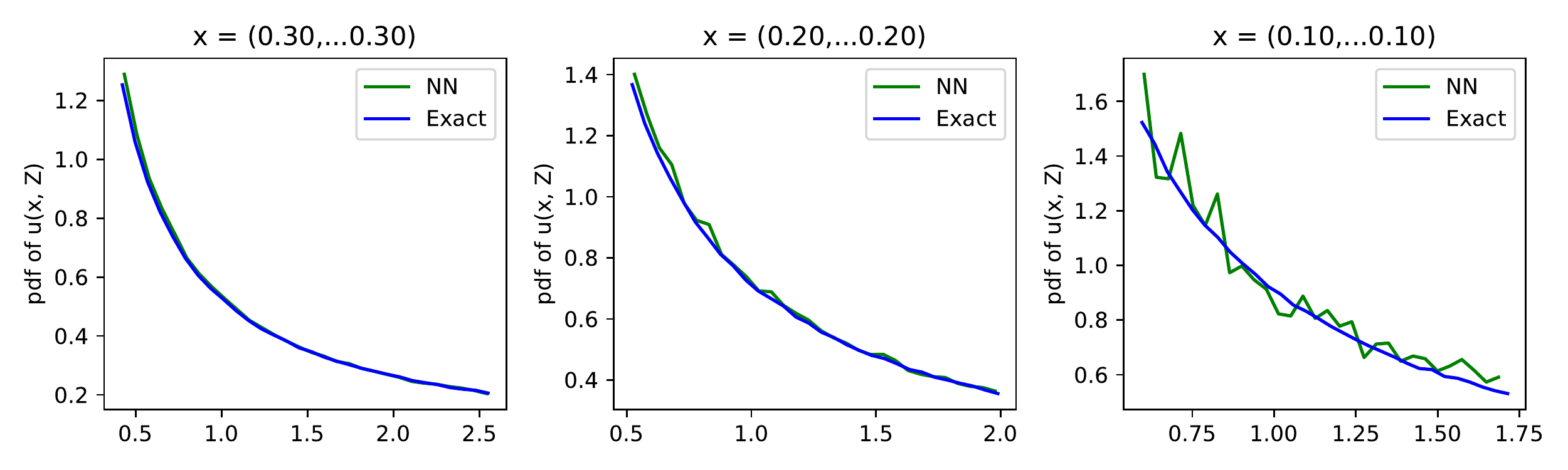}
\caption{Marginal distributions of the solution to the problem in Section~\ref{sec:overdamped} with $d = 10$.  Left: the pdf of $u_{\theta}(x, Z)$ at $x = (0.3,  \ldots, 0.3) \in \mathbb{R}^{10}$; 
Center: the pdf of $u_{\theta}(x, Z)$ at $x = (0.2,  \ldots, 0.2) \in \mathbb{R}^{10}$; Right: the pdf of $u_{\theta}(x, Z)$ at $x = (0.1,  \ldots, 0.1) \in \mathbb{R}^{10}$.}
\label{fig:Langevin_10d}
\end{figure}


\section*{Summary and Conclusion}

In this article, we have presented a DNN-based method for solving
stochastic variational problems through a combination of Monte Carlo
sampling and DNN approximation. As an application, we mainly focused
on the stochastic variational problem involving a quadratic
functional.

Our method differs from traditional approaches for solving stochastic
variational problems.  First and foremost, the utilization of the DNN
approximation and Monte Carlo sampling allows to avoid a
discretization of the physical domain (e.g., finite element and finite
volume) that often incurs an exponential growth of computational cost
with the number of dimensions.  Therefore, the method is applicable to
problems with high dimensions in both physical space and stochastic
space.  Second, the stochastic variational formulation combined with
DNN offers a natural loss function readily optimized by SGD methods.
Moreover, the variational approach is well known to retain certain
structures of the original problem and hence it is particularly
advantageous when applied to structure preserving problems.

Despite the above benefits, our method does suffer from several
limitations.  From the theoretical viewpoint, the lack of convergence
result of our method makes the hyperparameter tuning ad-hoc.  This
issue becomes particularly challenging for high-dimensional problems.
From the practical viewpoint, the boundary condition penalization
incorporated in the loss function alters the ground truth and hence
limits the superior accuracy of the method.  
Lastly, the
Monte Carlo sampling over a general domain $D$ can be quite
challenging as it involves a constrained sampling over $D$ and
$\partial D$. For a complex domain $D$ (e.g., with holes), the
rejection sampling may become infeasible, especially in high
dimension.  In this case, an application of the constrained Markov
chain Monte Carlo may be required~\cite{zhang2020ergodic,
  ciccotti2008projection, lelievre2012langevin, zappa2018monte,
  lelievre2019hybrid}.


\bibliography{stochastic_deep_Ritz}

\appendix

\section{The quadratic Lagrangian setting}\label{sec:appendix}
An important class of problems related to the stochastic variational problem~\eqref{eqn:stochastic-var-form}
is the random coefficient elliptic equation~\eqref{eqn:PDE-finite-noise}.
To recast~\eqref{eqn:PDE-finite-noise} into a stochastic variational problem~\eqref{eq:sochastic_functional_finite_dim},
we first set up an appropriate function space for the problem.  We
introduce the physical space $H_0^1(D)$, i.e., the Sobolev space of
functions with weak derivatives up to order $1$ and vanishing on the
boundary.  To define the function space for $Z$, we assume that $Z$
admits a probability density function $\rho(z)$ so that
\[
\mathbb{E}[g(Z(\omega))] = \int_{\Gamma} g(z) \rho(z)\, dz 
\]
for any measurable function $g: \Gamma \to \mathbb{R}$. 
Then we define the stochastic space $L_{\rho}^2(\Gamma)$, i.e.,
\[
L_{\rho}^2(\Gamma) = \left\{g: \Gamma \to \mathbb{R} ~\left| ~\int_{\Gamma} |g(z)|^2 \rho(z)\, dz  < \infty\right.\right\},
\]
which is the $\rho$-weighted $L^2$ space over $\Gamma$. 
The solution $u(x, Z)$ to \eqref{eqn:PDE-finite-noise} is thus defined in the space 
\begin{equation}\label{eqn:Hilbert-space}
L^2(\Gamma, H_0^1(D))
= \left\{v: D \times \Gamma \to \mathbb{R} \left|~ v \textrm{~is strongly measurable and~}   \|v\|_{H_0^1(D)} \in L_{\rho}^2(\Gamma)  \right.\right\},
\end{equation}
where the norm $\|\cdot\|_{L^2(\Gamma, H_0^1(D))}$ is induced by the 
inner product:
\begin{equation}\label{eqn:inner-product}
\begin{split}
\langle u, v \rangle  
&= \int_{\Gamma}\int_D u(x, z) v(x, z) \rho(z)\, dx \, dz  + \int_{\Gamma}\int_D \nabla u(x, z)  \cdot \nabla v(x, z) \rho(z)\, dx \, dz\\
&= \mathbb{E}\left[\int_D u(x, Z) v(x, Z)\, dx \right]  + \mathbb{E}\left[\int_D \nabla u(x, Z)  \cdot \nabla v(x, Z)\, dx \right].
\end{split}
\end{equation}
Under certain conditions, a direct application of the Lax-Milgram
theorem immediately implies the well-posedness of the following
stochastic weak form of~\eqref{eqn:PDE-finite-noise}:
\begin{equation}\label{eqn:weak-form}
    \mathbb{E}\left[\int_D \kappa(x, Z) \nabla u(x, Z) \cdot \nabla v(x, Z) - f(x, Z) v(x, Z) \, dx\right]=0, \quad \forall v \in L^2(\Gamma, H_0^1(D)).
\end{equation}
The above stochastic weak form is equivalent to a min-max problem
which involves a saddle points searching. Therefore, it is possible to
solve the problem using deep learning frameworks such as the
generative adversarial network (GAN) and its
variants~\cite{goodfellow2014generative, arjovsky2017wasserstein}.
However, it is well known that training of GAN is a difficult task
often requiring fine tuning of hyper-parameters.  The training becomes
even more challenging in the context of PDEs.  We refer the reader
to~\cite{zang2020weak} for a recent work on solving deterministic PDEs
using GAN.

Owing to the difficulty of training GAN models based on the stochastic
weak form, we alternatively reformulate~\eqref{eqn:PDE-finite-noise}
into the following stochastic variational problem
\begin{equation}\label{eqn:Ritz-form-precise-space}
\min_{u \in L^2(\Gamma, H_0^1(D))}  J(u) \quad \textrm{with}\quad J(u) =  \mathbb{E}\left[\int_D \frac{1}{2}\kappa(x, Z) |\nabla u(x, Z)|^2  - f(x, Z) u(x, Z)\, dx\right].
\end{equation}
This stochastic variational problem can be viewed as an extension of
the Ritz formulation (or Dirichlet's principle) to the stochastic
setting. The following theorem provides the theoretical justification
of the stochastic variational reformulation.
\begin{theorem}\label{thm:existence-and-uniqueness}
Under the following two assumptions:
\begin{enumerate}
\item[A1.]
The diffusivity coefficient $\kappa(x, Z)$ is uniformly bounded and uniformly coercive, i.e.,
there exist constants $0 < \kappa_{\min} \leq \kappa_{\max}$ such that for almost all $\omega \in \Omega$, 
\[\mathbb{P}\left(\omega \in \Omega : \kappa_{\textrm{min}} \leq \kappa(x, Z(\omega)) \leq \kappa_{\textrm{max}}, ~\forall x \in \bar{D} \right) = 1.\]
\item[A2.] The forcing $f(x, Z)$ satisfies the following integrability condition:
\[
\mathbb{E}\left[\int_D |f(x, Z)|^2 \, dx\right] < \infty.
\] 
\end{enumerate}
Then the stochastic variational problem~\eqref{eqn:Ritz-form-precise-space} admits a minimizer $u^* \in L^2(\Gamma, H_0^1(D))$.
Furthermore, $u^*$ satisfies the stochastic weak form~\eqref{eqn:weak-form}.
\end{theorem}

\begin{proof}
To show the existence of a minimizer $u^*$ to~\eqref{eqn:Ritz-form-precise-space}, it is sufficient to show that 1) there exists a minimizing sequence $u_n(x, Z)$ (up to a subsequence) that converges weakly to some $u^* \in L^2(\Gamma, H_0^1(D))$ and 2) the functional $J(u)$ is (weakly) lower semicontinuous~\cite{dacorogna2007direct}. 

Denote $J^* = \inf \{J(u) ~|~ u \in L^2(\Gamma, H_0^1(D))\}$. 
Note that $J^* > -\infty$ and hence there exists a minimizing sequence $u_n$ such that $J(u_n) \to J^*$. We shall show that there exists $u^* \in L^2(\Gamma, H_0^1(D))$ such that $J^* = J(u^*)$.

For weak convergence of the sequence $u_n(x, Z)$, it suffices to uniformly bound $u_n$ in $L^2(\Gamma, H_0^1(D))$.
To this end, note that by the assumption on $\kappa$, 
\[
\frac{1}{2}\kappa(x, Z) |\nabla u_n(x, Z)|^2 - f(x, Z) u_n(x, Z) \geq \frac{1}{2} \kappa_{\text{min}}|\nabla u_n(x, Z)|^2 - f(x, Z) u_n(x, Z).
\]
Integrating over $D$ and taking expectation of both sides lead to
\[
J(u_n) \geq \frac{1}{2} \kappa_{\text{min}}\mathbb{E}\left[\int_D |\nabla u_n(x, Z)|^2 \, dx  \right] - \mathbb{E}\left[\int_D f(x, Z) u_n(x, Z) \, dx\right].
\]
Invoking the assumption on $\kappa$ and the Cauchy-Schwarz inequality leads to
\[
J(u_n) \geq C_1\mathbb{E}\left[||\nabla u_n(\cdot, Z)||_{L^2(D)}^2\right] - \mathbb{E}\left[\| f(\cdot, Z)\|_{L^2(D)}^2\right]^{\frac{1}{2}} \mathbb{E}\left[\|u_n(\cdot, Z)\|_{L^2(D)}^2 \right]^{\frac{1}{2}},
\]
where $C_1 = \kappa_{\min}/2$.
By the assumption on $f$ and the Poincar\'{e}'s inequality, there exists a constant $C_2 > 0$ (which depends on $f$ and the domain $D$) such that
\[
J(u_n) \leq C_1\mathbb{E}\left[||\nabla u_n(\cdot, Z)||_{L^2(D)}^2\right] 
- C_2  \mathbb{E}\left[\|\nabla u_n(\cdot, Z)\|_{L^2(D)}^2 \right]^{\frac{1}{2}}.
\]
Now note that $J(u_n)$ is uniformly bounded (in $n$) since it is a minimizing sequence of~\eqref{eqn:Ritz-form-precise-space}, which implies 
$\mathbb{E}\{\|\nabla u_n(\cdot, Z)\|_{L^2(D)}^2 \}$
is uniformly bounded (in $n$) and hence 
\[\sup_n\mathbb{E}\{\|u_n(\cdot, Z)\|_{H_0^1(D)}^2 \} < \infty.\] 
Therefore, there exists a subsequence, still denoted by $u_n$, that converges weakly to some $u^* \in L^2(\Gamma, H_0^1(D))$.

Next, we justify that $J(u)$ is (weakly) lower semicontinuous. 
For each $z \in \mathbb{R}^K$,
define the functional $L_z: D \times \mathbb{R} \times \mathbb{R}^n \to \mathbb{R}$ by 
\[L_z(x, u, \xi) = \frac{1}{2}\kappa(x, z) |\xi|^2 + f(x, z) u.\]
Since $(u, \xi) \mapsto L_z(x, u, \xi)$ is convex for every $x$ and $z$, 
we have for all $x \in D$ and almost all $\omega \in \Omega$,  
\begin{equation*}
\begin{split}
L_{Z(\omega)}(x, u_n, \nabla u_n) \geq L_{Z(\omega)}(x, u^*, \nabla u^*) &+ \partial_u L_{Z(\omega)}(x, u^*, \nabla u^*) (u_n - u^*) \\
&+ \partial_{\xi} L_{Z(\omega)}(x, u^*, \nabla u^*) \cdot (\nabla u_n - \nabla u^*).
\end{split}
\end{equation*}
Integrating both sides of the above inequality and then taking expectations lead to 
\[
J(u_n) \geq J(u^*) + \mathbb{E}\left[\int_D f(x, Z) (u_n - u^*)\, dx\right]
+
\mathbb{E}\left[\int_D \kappa(x, Z) \nabla u^* \cdot (\nabla u_n - \nabla u^*)\, dx\right].
\]
Since $u_n$ converges weakly to $u^*$ in $L^2(\Gamma, H_0^1(D))$ and both $f(x, Z)$ and $\kappa(x, Z) \nabla u^*$ are in $L^2(\Gamma, H_0^1(D))$, we immediately have
\[
\liminf_{n\to \infty} J(u_n) \geq J(u^*).
\]
Therefore, $J(u)$ is (sequentially weak) lower semicontinuous and hence 
$u^*$ is a minimizer by the direct method of calculus of variation. 
It remains to be shown that $u^*$ satisfies the weak form~\eqref{eqn:weak-form}.
To this end, we consider the functional $J$ evaluated at $u^* + \epsilon v$ for every $v \in L^2(\Gamma, H_0^1(D))$. 
A simple calculation shows that the Gateaux derivative satisfies
\begin{equation*}
\begin{split}
\lim_{\epsilon \to 0}
\frac{1}{\epsilon}(J(u^* + \epsilon v) - J(u^*))
=
\mathbb{E}\left[\int_D \kappa(x, Z) \nabla u^* \cdot \nabla v - f(x, Z) v \, dx\right].
\end{split}
\end{equation*}
Since $u^*$ is a
minimizer of $J$, the Gateaux derivative
\[
\left.\frac{d}{d\epsilon}\right|_{\epsilon = 0}J(u^* + \epsilon v) = 0,
\]
and hence the weak form~\eqref{eqn:weak-form}.
\end{proof}




%

\end{document}